\documentclass[11pt]{amsart}
\usepackage[margin=1.1in]{geometry}
\usepackage{amsmath,amssymb,amsthm,mathtools}
\usepackage{enumitem}
\usepackage{booktabs}
\usepackage[colorlinks=true,linkcolor=blue!50!black,citecolor=blue!50!black,urlcolor=blue!50!black]{hyperref}
\usepackage{xcolor}

\newtheorem{theorem}{Theorem}
\newtheorem{lemma}[theorem]{Lemma}

\theoremstyle{definition}

\theoremstyle{remark}
\newtheorem{remark}[theorem]{Remark}

\newcommand{\N}{\mathbb{N}}
\newcommand{\Q}{\mathbb{Q}}
\newcommand{\E}{\mathbb{E}}
\newcommand{\PP}{\mathbb{P}}
\newcommand{\abs}[1]{\lvert #1\rvert}
\newcommand{\floor}[1]{\lfloor #1\rfloor}
\newcommand{\ceil}[1]{\lceil #1\rceil}
\newcommand{\cH}{\mathcal{H}}
\newcommand{\cM}{\mathcal{M}}
\newcommand{\cI}{\mathcal{I}}

\title{A negative answer to Erd\H{o}s Problem \#786}
\author{Shisheng Li}
\date{\today}

\begin{document}

\begin{abstract}
Call a set $A$ of positive integers \emph{admissible} if, whenever $a_1\cdots a_r=b_1\cdots b_s$ with $a_1,\dots,a_r$ distinct elements of $A$ and $b_1,\dots,b_s$ distinct elements of $A$, necessarily $r=s$. Erd\H{o}s asked whether admissible sets can have density $1-\varepsilon$ for every $\varepsilon>0$, and whether $\{1,\dots,N\}$ always contains an admissible subset of size $(1-o(1))N$. For the variant in which repetitions are allowed both questions were answered negatively by Erd\H{o}s, Ruzsa and S\'ark\"ozy and by Granville and Soundararajan; for products of distinct elements, the first question was answered only recently (with density bound $7/8$), and the second has remained open. We show that every admissible $A\subseteq\{1,\dots,N\}$ satisfies $\sum_{a\in A}1/a\le\frac12\log N+(\log\log N+2)^2$, and that there is an absolute constant $\eta>0$ such that every admissible $A\subseteq\{1,\dots,N\}$ has $\abs{A}<(1-\eta)N$ for all large $N$. Both questions therefore have negative answers. The proofs are elementary; the second rests on a coupling that replaces the largest divisor of an integer composed of small primes, which avoids the divisor-function losses inherent in counting quotients along a multiplication table. Both negative answers are formally verified in Lean~4 against the statements of the Formal Conjectures project.
\end{abstract}

\maketitle

\section{Introduction}

For a set $A$ of positive integers consider the property
\begin{equation}\tag{P}
S,T\subseteq A \text{ finite},\quad \prod_{s\in S}s=\prod_{t\in T}t\quad\Longrightarrow\quad\abs{S}=\abs{T}.
\end{equation}
Equivalently, an identity $a_1\cdots a_r=b_1\cdots b_s$ between \emph{distinct} elements $a_i$ of $A$ and distinct elements $b_j$ of $A$ forces $r=s$ (common elements of the two sides may be cancelled). We call such sets \emph{admissible}. Note that $1\notin A$, and that every subset of an admissible set is admissible. The integers $\equiv2\pmod 4$ form an admissible set of density $1/4$, since the $2$-adic valuation counts the factors on each side.

Erd\H{o}s~\cite{Er65,Er80} asked the following (Problem \#786 of~\cite{EP}).
\begin{enumerate}[label=(\roman*)]
\item Is there, for every $\varepsilon>0$, an admissible set of density greater than $1-\varepsilon$?
\item Is there always an admissible $A\subseteq\{1,\dots,N\}$ with $\abs{A}\ge(1-o(1))N$?
\end{enumerate}
Selfridge observed that the integers having exactly one prime factor, counted with multiplicity, from a suitable block of consecutive large primes (with reciprocal sum just below $1$) form an admissible set of density $1/e-\varepsilon$, and Tao observed (see~\cite{EP}) that the integers $n\le N$ with exactly one prime factor, counted with multiplicity, exceeding $N^{1/(1+\sqrt e)}$ form an admissible set of size $(1-c+o(1))N$ with $c=0.1715\ldots$.

If one allows repetitions (so that (P) is required for multisets), then a finite set is admissible precisely when there is a completely additive $F:\N\to\Q$ with $F\equiv1$ on $A$ (by linear duality), and both questions are settled: Erd\H{o}s, Ruzsa and S\'ark\"ozy~\cite{ERS73} obtained the density bound $1/2$, and Tao noted that a theorem of Granville and Soundararajan~\cite{GS01} gives $\abs{A}\le(1-c+o(1))N$ with the constant $c$ above, which is sharp. The formulation in~\cite[p.~114]{Er80}, however, explicitly concerns distinct elements; there Erd\H{o}s states that Ruzsa had shown that both answers are negative, and that an infinite admissible set has upper density $<1/e$. No proof has appeared, and the problem page~\cite{EP} remarks that Erd\H{o}s may have had the repetition-allowed results in mind; the distinct version is listed there as open. Recently Gessel~\cite{Gessel} posted a proof, accompanied by a Lean formalization, that an admissible set with natural density $\delta$ has $\delta\le7/8$, which answers (i); question (ii) remained open. The repetition-free condition is weaker, and there is no additive certificate: for instance $\{2,3,4\}$ is admissible but $2\cdot2=4$.

\begin{theorem}\label{thm:log}
For every $N\ge286$ and every admissible $A\subseteq\{1,\dots,N\}$,
\[
\sum_{a\in A}\frac1a\le\frac12\log N+(\log\log N+2)^2 .
\]
Consequently an admissible set has upper logarithmic density at most $1/2$, and natural density at most $1/2$ whenever the latter exists.
\end{theorem}

\begin{theorem}\label{thm:main}
There is an absolute constant $\eta>0$ such that for all sufficiently large $N$, every admissible $A\subseteq\{1,\dots,N\}$ satisfies $\abs{A}<(1-\eta)N$. One may take $\eta=e^{-5000}$ for $N\ge e^{200000}$.
\end{theorem}

Theorem~\ref{thm:main} answers (ii) negatively, and also (i): an admissible set of density $>1-\eta/2$ would contain admissible sets $A\cap[1,N]$ of size $\ge(1-\eta)N$. Theorem~\ref{thm:log} gives the stronger bound $1/2$ for (i), matching the repetition-allowed bound of~\cite{ERS73}. We make no attempt to optimize $\eta$; it seems natural to conjecture, following Tao, that the optimal deficit is $0.1715\ldots$, i.e.\ that the largest admissible subsets of $\{1,\dots,N\}$ have size $(1-0.1715\ldots+o(1))N$, in the distinct setting as well. Some mild numerical evidence points the same way: an exact computation for $N\le60$ (two independent programs, available with~\cite{repo}) shows that for $34\le N\le60$ the largest admissible subset of $\{1,\dots,N\}$ has the same size as in the repetition-allowed version, and for $39\le N\le 60$ an extremal set is $\{n\le N: n$ has exactly one prime factor $\ge5$, counted with multiplicity$\}$, a level set of a completely additive function.

\subsection*{Formal verification}
The negative answers to both questions, in the exact form stated in the Formal Conjectures project~\cite{FC} (\texttt{erdos\_786.parts.i} and \texttt{erdos\_786.parts.ii}, with the answer resolved to \texttt{False}), are proved in Lean~4 with Mathlib; the only analytic input beyond Mathlib is the explicit form of Mertens' theorems, taken (vendored) from the PrimeNumberTheoremAnd project~\cite{PNT}. The formal proof follows Section~\ref{sec:main} with unspecified constants. See Section~\ref{sec:lean}.

\subsection*{Notation}
$H(S)=\sum_{s\in S}1/s$, $H_j=\sum_{n\le j}1/n$; $v_p(n)$ is the exponent of $p$ in $n$; $\tau(n)$ the number of divisors; $p,q$ denote primes. For a positive rational $q\ne1$, a \emph{$q$-pair} in $A$ is a pair $(x,qx)$ with $x,qx\in A$, and a \emph{$q$-matching} is a family of $q$-pairs whose endpoints are pairwise distinct.

\section{Two combinatorial lemmas}

\begin{lemma}[Lifting]\label{lem:lift}
Let $a\in A$ and $a=q_1\cdots q_t$ with positive rationals $q_j\ne1$ (repetitions allowed). If for each $j$ the $q_j$-pairs in $A$ contain a matching of size at least $2t$, then $A$ is not admissible.
\end{lemma}
\begin{proof}
Choose $(x_j,y_j)$, $y_j=q_jx_j$, for $j=1,\dots,t$ in turn from the matching of $q_j$, avoiding $a$ and all endpoints chosen earlier. At step $j$ there are $2j-1$ forbidden integers, each lying in at most one pair of a matching, so a pair is available. Then $a,x_1,y_1,\dots,x_t,y_t$ are distinct elements of $A$ and $a\prod x_j=\prod y_j$, with $t+1\ne t$ factors. (If $t=0$ then $a=1$, which is excluded by (P).)
\end{proof}

\begin{lemma}\label{lem:paths}
Let $q>1$. Any $R$ distinct $q$-pairs contain a $q$-matching of size at least $R/2$.
\end{lemma}
\begin{proof}
The graph with edges $\{x,qx\}$ has maximum degree $2$ and is acyclic (values increase along $x\mapsto qx$), so it is a union of paths; take alternate edges.
\end{proof}

A first consequence: if $R_p=\#\{x:x,px\in A\}$ and $a=p_1\cdots p_k\in A$ with every $R_{p_i}\ge4\log_2N$, then greedily choosing fresh pairs $(x_i,p_ix_i)$ (at step $i$ at most $4i-2$ candidates are blocked) gives $a\prod x_i=\prod p_ix_i$. Since $R_p\ge\floor{N/p}-2(N-\abs A)$, this already forces all $y$-smooth integers with $y\approx N/(2(N-\abs{A}))$ out of $A$; but that yields only a deficit $N\exp(-(1+o(1))\sqrt{\log N\log\log N})$.

\section{Logarithmic density: proof of Theorem~\ref{thm:log}}

\begin{lemma}[Packing]\label{lem:packing}
Let $P$ be a finite set of primes and $S\subseteq\{1,\dots,N\}$ such that every element of $S$ is divisible by some $p\in P$, and $x,px$ never both lie in $S$ for $p\in P$. Then $H(S)\le\frac12H_N+\frac12$.
\end{lemma}
\begin{proof}
Fix $m$ and write $m=b\prod_{i\le r}p_i^{e_i}$ with $p_1,\dots,p_r$ the primes of $P$ dividing $m$ and $(b,\prod p_i)=1$. For each $c\mid b$ the divisors $c\prod p_i^{f_i}$, $0\le f_i\le e_i$, form a box $B=\prod\{0,\dots,e_i\}$; the exponent vectors of members of $S$ avoid $0$ and contain no two vectors differing by $1$ in one coordinate. Pair $B$ into adjacent pairs along a coordinate with $e_i$ odd if there is one; otherwise pair $B\setminus\{0\}$ along the first nonzero coordinate via $(1,2),(3,4),\dots$. Hence $\#\{d\mid m:d\in S\}\le\tau(m)/2$. Summing over $m\le N$ gives $\sum_{s\in S}\floor{N/s}\le\frac12\sum_{d\le N}\floor{N/d}$; writing $r_d=N/d-\floor{N/d}\in[0,1)$,
\[
N H(S)\le\tfrac N2H_N+\tfrac12\Big(\sum_{s\in S}r_s-\sum_{d\notin S}r_d\Big)\le\tfrac N2H_N+\tfrac N2 .\qedhere
\]
\end{proof}

\begin{proof}[Proof of Theorem~\ref{thm:log}]
Let $h=\floor{\log_2N}$ and call a prime $p\le N$ \emph{bad} if $R_p<4h$; let $P$ be the (finite) set of bad primes. Every $a\in A$ has a bad prime factor, by the consequence of Lemma~\ref{lem:lift} noted above (for $a=p_1\cdots p_k$ with no bad factor, $k\le h$ and fresh pairs exist). Let $E=\{px:p\in P,\ x,px\in A\}$ and $S=A\setminus E$. Then $S$ satisfies the hypotheses of Lemma~\ref{lem:packing}, so $H(S)\le\frac12H_N+\frac12$, while
\[
H(E)\le\sum_{p\in P}\frac1p\sum_{x:\,x,px\in A}\frac1x\le H_{4h-1}\sum_{p\le N}\frac1p .
\]
By Rosser--Schoenfeld~\cite[Thm.~5]{RS62}, $\sum_{p\le N}1/p\le\log\log N+1$ for $N\ge286$; also $H_{4h-1}\le1+\log(4\log N/\log2)<\log\log N+3$ and $\frac12H_N+\frac12\le\frac12\log N+1$. With $\ell=\log\log N$ we get $H(A)\le\frac12\log N+1+(\ell+3)(\ell+1)=\frac12\log N+(\ell+2)^2$. The statements about densities follow by applying this to $A\cap[1,X]$ and, for natural density, by partial summation.
\end{proof}

\begin{remark}
Harmonic mass is insensitive to $[N/K,N]$ (the interval $[N/\log N,N]$ has $(1-o(1))N$ elements but harmonic mass $\approx\log\log N$), so Theorem~\ref{thm:log} does not bear on question (ii).
\end{remark}

\section{Proof of Theorem~\ref{thm:main}}\label{sec:main}

\subsection{Why a new idea is needed}
A natural approach to (ii) is to connect a large prime $p$ to $1$ through ratios $v/u$ close to $1$ between integers $u,v\le\sqrt N$, and to charge each missing ratio edge to the missing integers $mt$ or $nt$ occurring in the ``trials'' $(mt,nt)$, $m\in(K,2K]$. Carried out in dyadic bands, this charges each failed trial to a divisor of a missing integer, so missing integers are weighted by their number of divisors; one obtains only a deficit bound with a polylogarithmic loss, $N-\abs A\gg N/(\log N)^{C}$. The loss is intrinsic to this kind of charging. By Ford's theorem~\cite[Thm.~1]{Ford08}, uniformly for $3\le K\le\sqrt N$ the number of $n\le N$ having a divisor in $(K,2K]$ is $\asymp N(\log K)^{-\delta_{\rm F}}(\log\log K)^{-3/2}$ with $\delta_{\rm F}=1-\frac{1+\log\log2}{\log2}$, which is $o(N)$ as $K\to\infty$; yet deleting these integers kills every trial $(mt,nt)$ with $m\in(K,2K]$. The proof below avoids this by pairing $pz$ with a partner obtained from $z$ through a \emph{canonical} factorization, so that every integer is hit with essentially uniform probability.

\subsection{Set-up}
Fix $N\ge N_1=\ceil{e^{200000}}$ and $A\subseteq\{1,\dots,N\}$; let $E=\{1,\dots,N\}\setminus A$, $D=\abs E$, $\delta=D/N$, and suppose $\delta\le\eta=e^{-5000}$. Put
\[
P=N^{1/32},\qquad M=\ceil{4(\log N)^2},\qquad c_0=e^{-1400},\qquad p_0=e^{3200}.
\]
A prime $p\le N$ is \emph{heavy} if $\abs{E\cap p\N}\ge N/(16p)$; let $\cH$ be the set of heavy primes and $S=\sum_{p\in\cH}1/p$. A rational $q\ne1$ is \emph{rich} if the $q$-pairs in $A$ contain a matching of size $M$; $q$ is rich iff $q^{-1}$ is. For a set $Q$ of primes, $k_Q(z)$ denotes the largest divisor of $z$ all of whose prime factors lie in $Q$.

\subsection{Step 1: heavy primes}
\begin{lemma}\label{lem:heavy}
$S\le3072\,\delta$. In particular $S+\eta<c_0/2$.
\end{lemma}
\begin{proof}
Let $\omega(n)=\#\{p\in\cH:p\mid n\}$. Averaging over $n\le N$, $\E\omega\ge S-\abs\cH/N$ and $\E\omega^2\le S+S^2$, so $\E(\omega-S)^2\le S+2S\abs{\cH}/N\le3S$. By Cauchy--Schwarz $\sum_{n\in E}\omega(n)\le DS+\sqrt{3DNS}$, whereas heaviness gives $\sum_{n\in E}\omega(n)=\sum_{p\in\cH}\abs{E\cap p\N}\ge NS/16$. Hence $(\frac1{16}-\delta)S\le\sqrt{3\delta S}$, i.e.\ $S\le3\delta/(\frac1{16}-\delta)^2\le3072\delta$.
\end{proof}

\subsection{Step 2: smooth numbers}
\begin{lemma}\label{lem:smooth}
If $y\ge e^{200}$ and $y^{16}\le X\le y^{56}$, then at least $c_0X$ integers in $[X/2,X]$ are $y$-smooth.
\end{lemma}
\begin{proof}
Let $u=\log X/\log y$, $a=(u-1)/112$, $b=(u-\frac12)/112$, so $\frac18<a<b<\frac12$ and $b/a\ge\frac{111}{110}$. By~\cite[Thm.~5]{RS62}, $\abs{\sum_{p\le z}1/p-\log\log z-B_1}\le1/\log^2z$ for $z\ge286$, whence the primes $\cI$ in $(y^a,y^b]$ satisfy $\sum_{q\in\cI}1/q\ge\log\frac{111}{110}-\frac{128}{200^2}>\frac1{200}$. For an ordered $112$-tuple from $\cI$ with product $Q$ we have $X/y<Q\le X/\sqrt y$, and the $\ge X/(3Q)$ integers $r\in[X/(2Q),X/Q]$ give $y$-smooth $n=rQ\in[X/2,X]$. There are at least $\frac X3(\sum_{q\in\cI}1/q)^{112}$ such representations, and each $n\le y^{56}$ has at most $448$ prime factors exceeding $y^{1/8}$, hence at most $448^{112}$ representations. So there are at least $X/(3\cdot89600^{112})>e^{-1400}X$ such $n$.
\end{proof}

\subsection{Step 3: the $Q$-part is usually small}
\begin{lemma}\label{lem:kQ}
Let $y\ge e^{200}$, $Q$ a set of primes $\le y$, and $z$ uniform in $[1,T]$. Then $\PP(k_Q(z)>y^{40})\le\frac1{10}$.
\end{lemma}
\begin{proof}
Chebyshev's bound $\vartheta(t)<3t$ (from $\prod_{n<p\le2n}p\mid\binom{2n}n$) and partial summation give $\sum_{q\le y}\frac{\log q}{q-1}\le3\log y+7\le4\log y$. Hence $\E\log k_Q(z)=\sum_{q\in Q}\sum_{j\ge1}(\log q)\floor{T/q^j}/T\le4\log y$, and Markov's inequality gives the claim.
\end{proof}

\subsection{Step 4: small primes}
Every prime $p\le p_0$ is rich: there are at least $\floor{N/p}-2D\ge N/(2p_0)$ $p$-pairs, hence a matching of size $N/(4p_0)\ge M$ by Lemma~\ref{lem:paths}. We also use $N^{57/64}/8\ge M$ and $P\ge p_0$, valid for $\log N\ge200000$.

\subsection{Step 5: replacing the largest smooth part}
\begin{lemma}\label{lem:core}
Let $p$ be a non-heavy prime with $p_0<p\le P$, $y=p^{1/16}$, and $Q$ the set of non-heavy primes $\le y$. There are $Q$-smooth integers $k,m$ with $k\le y^{40}$ and $pk/2\le m<pk$ such that $pk/m$ is rich.
\end{lemma}
\begin{proof}
For $Q$-smooth $k\le y^{40}$ let $\cM_k$ be the set of $Q$-smooth $m\in[pk/2,pk]$. Since $y^{16}\le pk\le y^{56}$, Lemma~\ref{lem:smooth} gives $c_0pk$ $y$-smooth integers in $[pk/2,pk]$, of which at most $pkS$ are divisible by a heavy prime; by Lemma~\ref{lem:heavy}, $\abs{\cM_k}\ge\gamma pk$ with $\gamma=c_0/2$. As $p\notin Q$, $pk\notin\cM_k$.

Let $T=\floor{N/p}$ and let $z$ be uniform in $[1,T]$; write $z=ku$ with $k=k_Q(z)$. If $k\le y^{40}$, pick $m\in\cM_k$ uniformly and form the pair $(mu,\,pku)=(mu,pz)$, both $\le N$. Failures:
(a) $k>y^{40}$, probability $\le\frac1{10}$ by Lemma~\ref{lem:kQ};
(b) $pz\in E$, probability $<\frac{N/(16p)}{T}\le\frac18$ as $p$ is not heavy;
(c) $mu\in E$. Given $b=mu$, $m$ is the largest $Q$-smooth divisor of $b$ (as $u$ has no prime factor in $Q$), so $m$ and $u$ are determined by $b$, and $z=ku$ with $m/p\le k\le2m/p$. Hence
\[
\PP(mu=b)\le\frac1{T\gamma p}\sum_{m/p\le k\le2m/p}\frac1k\le\frac{2}{T\gamma p}\le\frac4{\gamma N},
\]
and $\PP(mu\in E)\le4\delta/\gamma\le8e^{-3600}<\frac18$.
So with probability $>\frac12$ both $mu$ and $pz$ lie in $A$. For $Q$-smooth $k\le y^{40}$ and $m\in\cM_k$ let $s_{k,m}$ be the number of $u$ with $ku\le T$, $k_Q(ku)=k$, $mu\in A$ and $pku\in A$. The success probability equals $\frac1T\sum_k\frac1{\abs{\cM_k}}\sum_{m\in\cM_k}s_{k,m}\le\frac1T\sum_k\max_{m}s_{k,m}$, so $\sum_k\max_ms_{k,m}>T/2$; as there are at most $y^{40}$ values of $k$, some $(k,m)$ has $s_{k,m}\ge T/(2y^{40})\ge N/(4p^{7/2})\ge N^{57/64}/4$; the corresponding $(mu,pku)$ are distinct $(pk/m)$-pairs, and Lemma~\ref{lem:paths} shows that $pk/m$ is rich.
\end{proof}

The uniqueness of the decomposition $z=k_Q(z)\cdot u$ is what removes the divisor-function weight: with an arbitrary divisor in place of $k_Q(z)$, a given $b$ would be reached through each of its divisors.

\subsection{Step 6: short words}
\begin{lemma}\label{lem:words}
Every non-heavy prime $p\le P$ is a product of at most $16(\log p)^2$ rich rationals and inverses of rich rationals.
\end{lemma}
\begin{proof}
Induct on $p$. For $p\le p_0$ use Step 4. Otherwise take $k,m$ from Lemma~\ref{lem:core}; their prime factors are non-heavy and $\le p^{1/16}<p$. Concatenating the words of the prime factors, a $Q$-smooth $v$ has a word of length at most $16\sum_qv_q(v)(\log q)^2\le(\log p)\log v$. Writing $p=\frac{pk}m\cdot m\cdot k^{-1}$ and using $k\le p^{5/2}$, $m\le pk$, the length is at most $1+(\log p)(\log m+\log k)\le1+6(\log p)^2\le16(\log p)^2$.
\end{proof}

\subsection{Step 7: conclusion}
Apply Lemma~\ref{lem:smooth} with $y=P$, $X=N$: at least $c_0N$ integers in $[N/2,N]$ are $P$-smooth; at most $NS$ are divisible by a heavy prime and at most $D$ lie in $E$, so some $a\in A\cap[N/2,N]$ has all its prime factors non-heavy and $\le P$. By Lemma~\ref{lem:words}, $a=q_1\cdots q_t$ with each $q_j$ rich (all rich ratios are $\ne1$ by definition) and $t\le16(\log P)\log a\le\frac12(\log N)^2$, so $M\ge2t$, and Lemma~\ref{lem:lift} shows that $A$ is not admissible. This proves Theorem~\ref{thm:main}. \qed

\section{Formalization}\label{sec:lean}
The Lean~4 development~\cite{repo} (Lean \texttt{v4.33.1}, Mathlib commit \texttt{0df444a}) depends only on Mathlib; the explicit Mertens estimates
$\sum_{p\le x}\frac{\log p}{p}=\log x+O(1)$ and $\sum_{p\le x}\frac1p=\log\log x+M+O(1/\log x)$, with explicit error terms, are vendored from PrimeNumberTheoremAnd~\cite{PNT} and replace~\cite{RS62}. It proves:
\begin{itemize}
\item \texttt{Erdos786.main}: there is $\eta>0$ such that for all large $N$ every $A\subseteq\{1,\dots,N\}$ with $\abs A\ge(1-\eta)N$ contains distinct $a_1,\dots,a_r,b_1,\dots,b_s$ with $r\ne s$ and $\prod a_i=\prod b_j$. The proof follows Section~\ref{sec:main} with unspecified constants: Lemma~\ref{lem:lift} and a greedy version of Lemma~\ref{lem:paths} (matching of size $\ge R/3$), Step~1 (\texttt{heavy\_recip\_sum\_le}), Step~2 (\texttt{smooth\_count}), Step~3 (\texttt{qPart\_tail}), Step~5 in counting form (\texttt{core\_pairs}), Step~6 (\texttt{short\_word}) and the assembly (\texttt{main\_of}).
\item \texttt{erdos\_786.parts.ii} and \texttt{erdos\_786.parts.i} of~\cite{FC}, in the form \texttt{False} $\leftrightarrow$ (statement), where the statements and all definitions they use (\texttt{Set.IsMulCardSet}, \texttt{Set.HasDensity}, \texttt{Set.partialDensity}) are copied verbatim. Part~(i) is deduced from \texttt{main} as explained after Theorem~\ref{thm:main}.
\end{itemize}
\texttt{\#print axioms} reports only \texttt{propext}, \texttt{Classical.choice} and \texttt{Quot.sound}. Theorem~\ref{thm:log} is not formalized.

\section*{Acknowledgements}
This work was carried out with substantial assistance from AI systems (OpenAI's GPT models and Anthropic's Claude), including the development of the proofs and the Lean formalization. The author takes full responsibility for the content. We thank T.~Tao for the discussion on the problem page~\cite{EP}.

\end{document}